\documentclass[reqno,12pt,a4paper]{amsart}

\usepackage{eurosym}
\usepackage{amsfonts}
\usepackage{amssymb}
\usepackage{amsmath}
\usepackage{latexsym}
\usepackage{float,subfigure,color,subfigure}
\usepackage{graphicx}
\usepackage{amsthm}
\usepackage{verbatim}

\theoremstyle{plain}
\newtheorem{theorem}{\bf Theorem}[section]
\newtheorem{lemma}[theorem]{\bf Lemma}

\newcommand{\R}{\mathbb{R}}

\title
[Well-posedness]
{
	On well-posedness of a dispersive system
	of the Whitham--Boussinesq type
}

\date{\today}

\author[Dinvay]{Evgueni Dinvay}

\address
{
	{\texttt{evgueni.dinvay@math.uib.no}},
	Department of Mathematics, University of Bergen,
	Postbox 7800, 5020 Bergen, Norway.
}

\begin{document}

\begin{abstract}
The initial-value problem for a particular bidirectional Whitham system
modelling surface water waves is under consideration.
This system was recently introduced in \cite{Dinvay_Dutykh_Kalisch}.
It is numerically shown to be stable and a good approximation
to the incompressible Euler equations.
Here we prove local in time well-posedness.
Our proof relies on an energy method and a compactness argument.
In addition some numerical experiments,
supporting the validity of the system as an asymptotic model
for water waves, are carried out.
\end{abstract}

\maketitle

\section{Introduction}
\setcounter{equation}{0}

We regard the Cauchy problem for the system
that in non-dimensional variables has the form
\begin{align}
\label{sys1}
	\eta_t &=
	- v_x - i \tanh D (\eta v)
	, \\
\label{sys2}
	v_t &=
	- i \tanh D \eta - i \tanh D v^2 / 2
\end{align}
where $D = -i \partial_x$
and so $\tanh D$ is a bounded self-adjoint operator in
$L_2(\mathbb R)$.
The system models the two-dimensional water wave problem
for an inviscid incompressible flow.
As usual $\eta$ denotes the surface elevation.
Its dual variable $v$ roughly speaking has the meaning
of the surface fluid velocity.

Equations \eqref{sys1}-\eqref{sys2} appeared in literature
recently as an alternative to other linearly fully dispersive models
able to describe two wave propagation \cite{Dinvay_Dutykh_Kalisch}.
Those models capture many interesting features of
the full water waves problem and are in a good agreement with
experiments \cite{Carter}.
As to well-posedness, the existing results for them are not satisfactory.
For example, the system regarded in \cite{Pei_Wang} is
locally well posed if only an additional non-physical condition
$\eta \geqslant C > 0$ is imposed.
This system is probably ill-posed for large data
if one removes the assumption $\eta > 0$.
An heuristic argument is given in \cite{Klein_Linares_Pilod}.
This is not a problem for System \eqref{sys1}-\eqref{sys2}.

Another important property of System \eqref{sys1}-\eqref{sys2}
is its Hamiltonian structure.
Indeed, regarding the functional
\begin{equation*}
	\mathcal H(\eta, v)  = \frac 12 \int_\R
	\left(	
		\eta^2 + v \frac{D}{\tanh D} v
		+ \eta v^2
	\right)
	dx
\end{equation*}
Equations \eqref{sys1}-\eqref{sys2} can be rewritten in the form
\[
	\partial_t (\eta, v)^T = J \nabla \mathcal H(\eta, v)
\]
with the skew-adjoint matrix
\[
	J
	=
	\begin{pmatrix}
		0 & - i \tanh D
		\\
		- i \tanh D & 0
	\end{pmatrix}
	.
\]
In particular, $\mathcal H$ is a conserved quantity.
Thus Equations \eqref{sys1}-\eqref{sys2} provide an example
of a nonlinear Hamiltonian system that is locally well posed.

In some sense \eqref{sys1}-\eqref{sys2} can be regarded
as a regularization of the
system introduced in \cite{Hur_Pandey}.
Indeed, if one formally admits that
$\tanh D \thicksim D$ for small frequencies, then
substituting $D$ instead of $\tanh D$ to the nonlinear
part of System \eqref{sys1}-\eqref{sys2} one arrives
to the system regarded in \cite{Hur_Pandey}.
Such approximation is in line with the long wave framework,
when we keep all dispersive terms in the linear part
and exactly first dispersive term untouched in the nonlinear part.
That is formally justified due to smallness of regarded water waves.
Changing variables and admitting $\tanh D \thicksim D$ in
nonlinear part, as explained in \cite{Dinvay_Dutykh_Kalisch},
one can arrive to the system studied in \cite{Pei_Wang}.
It is also worth to notice that for the system regarded
in \cite{Hur_Pandey} the Benjamin--Feir instability
of periodic travelling waves is proved.
If one in addition formally discard the term $\eta \partial_x u$
in the system given in \cite{Hur_Pandey},
then a new alternative system turns out to be locally well-posed
and features wave breaking \cite{Hur_Tao}.

In addition it is worth to notice that
System \eqref{sys1}-\eqref{sys2} outperforms
other bidirectional Whitham models both in the sense 
of numerical stability and accuracy of approximation of
Euler equations \cite{Dinvay_Dutykh_Kalisch}.
This is might not be surprising since in the nonlinear part
of Equations \eqref{sys1}-\eqref{sys2} we have a bounded operator.
However, if one tries to diagonalise the system
then one will encounter a fractional derivative
$|D|^{1/2}$ both in the linear and nonlinear parts.
So further considerations turn out to be not completely straightforward.

Finally, let us formulate the main result.
We stick to the usual notations of Sobolev spaces
$H^s = H^s(\mathbb R)$ with the norm defined via Fourier transform.

\begin{theorem}
	For any $\eta_0 \in H^{1/2}(\mathbb{R})$ and
	$v_0 \in H^1(\mathbb{R})$
	there exists a positive time $T > 0$ depending only on the norm
	\(
		\lVert \eta_0 \rVert _{H^{1/2}}
		+ \lVert v_0 \rVert _{H^1}
	\)
	such that there exists unique solution
	\(
		( \eta, v ) \in C([0, T];
		H^{1/2}(\mathbb{R}) \times H^1(\mathbb{R}) )
	\)
	of System \eqref{sys1}-\eqref{sys2} with
	the initial data $( \eta_0, v_0 )$.
	Moreover, it depends continuously on the initial data.
\end{theorem}

In the following section a priori bound is established.
The complete proof of the existence would result
from a standard compactness argument implemented on
a regularised version of the system.
In the third section, we derive an estimate for
the difference of two solutions.
With this estimate in hand, one can prove the uniqueness
as well the continuity of the flow map.
In the end, the relevance of \eqref{sys1}-\eqref{sys2}
as an asymptotic model for water waves is supported
by numerical calculations. The latter demonstrate a good agreement
with the Euler equations.
%
%
\section{A priori estimate}
\setcounter{equation}{0}

Introduce a functional of the form
\begin{equation}
\label{norm_rest}
	B(\eta, v) = \frac 12 \int _{ \mathbb R }
	\left(
		\eta | D | \eta 
		+ v | D |^2 v
	\right)
	dx
\end{equation}
and a norm $E(\eta, v)$ of the view
\begin{equation}
\label{E_norm}
	E^2(\eta, v)
	= \frac 12 \lVert \eta \rVert _{L_2}^2
	+ \frac 12 \lVert v \rVert _{L_2}^2
	+ B(\eta, v)
\end{equation}
that is obviously equivalent to
\(
	\lVert \eta \rVert _{H^{1/2}}
	+ \lVert v \rVert _{H^1}
\).
Here the pair $\eta(x, t)$, $v(x, t)$ represents a possible solution of
System \eqref{sys1}-\eqref{sys2}.

\begin{lemma}[A priori estimate]
	Suppose
	$\eta(t) \in H^{1/2}(\mathbb{R})$ and
	$v(t) \in H^1(\mathbb{R})$ solving System \eqref{sys1}-\eqref{sys2}
	are defined on some interval including zero.
	Then there exist constants $C>0$ and $T>0$ such that
	\[
		E(t) \leqslant \frac{ E_0 e^{Ct} }{ 1 - E_0 (e^{Ct} - 1) }
	\]
	for any $t \in [0, T)$.
	Here $E(t)$ stands for $E(\eta(t), v(t))$ and $E_0 = E(0)$. 
\end{lemma}
\begin{proof}
Firstly, calculate the obvious derivative
\[
	\frac 12 \frac{d}{dt} \lVert \eta \rVert _{L_2}^2
	=
	\int \eta \eta_t
	=
	- \int \eta v_x - i \int \eta \tanh D(\eta v)
	\leqslant
	\lVert \eta \rVert _{L_2}
	\lVert \partial_x v \rVert _{L_2}
	+
	\lVert \eta \rVert _{L_2}^2
	\lVert v \rVert _{L_{\infty}}
\]
that follows from H\"older's inequality and boundedness of
operator $\tanh D$ in $L_2(\mathbb{R})$.
Similarly
\[
	\frac 12 \frac{d}{dt} \lVert v \rVert _{L_2}^2
	\leqslant
	\lVert \eta \rVert _{L_2}
	\lVert v \rVert _{L_2}
	+
	\lVert v \rVert _{L_2}^2
	\lVert v \rVert _{L_{\infty}}
	.
\]
Hence derivative of the first two terms in \eqref{E_norm}
is bounded as
\begin{equation}
\label{first_dE}
	\frac 12 \frac{d}{dt}
	\left(
		\lVert \eta \rVert _{L_2}^2
		+
		\lVert v \rVert _{L_2}^2
	\right)	
	\leqslant
	\left(
		\lVert \eta \rVert _{L_2}
		+
		\lVert \eta \rVert _{L_2}^2
		+
		\lVert v \rVert _{L_2}^2
	\right)
	\lVert v \rVert _{H^1}
	.
\end{equation}

Differentiating \eqref{norm_rest} with respect to $t$ obtain
\begin{multline}
\label{dB_multline}
	\frac {d}{ dt } B(\eta, v)
	=
	\int _{ \mathbb R }
	\left[
		\eta |D| \eta_t
		+ v | D |^2 v_t
	\right]
	dx
	=
	\\
	=
	\int _{ \mathbb R }
	\left[
		- \eta |D| \partial_x v
		- i \eta  |D| \tanh D (\eta v)
		- i v | D |^2 \tanh D \eta
		- i v | D |^2 \tanh D v^2 / 2
	\right]
	dx
\end{multline}
Note that $i |D|^2 \tanh D = \partial_x |D||\tanh D| $
and so combining the first and the third integral
in \eqref{dB_multline} gives
\[
	- \int \eta |D| \partial_x v
	- \int v \partial_x |D||\tanh D| \eta
	=
	\int v \partial_x |D| ( 1 - |\tanh D| ) \eta
	\leqslant
	C \lVert \eta \rVert _{L_2} \lVert v \rVert _{L_2}
\]
since operator $ \partial_x |D| ( 1 - |\tanh D| ) $
is obviously bounded.
Again applying $i |D| \tanh D = \partial_x ( |\tanh D| - 1 ) + \partial_x $
to the second part of the Integral \eqref{dB_multline} obtain
\[
	- i \int \eta  |D| \tanh D (\eta v)
	=
	\int \eta \partial_x ( 1 - |\tanh D| ) (\eta v)
	- \int \eta \partial_x (\eta v)
\]
where the first integral is bounded by
\(
	\lVert \eta \rVert _{L_2}^2 \lVert v \rVert _{L_{\infty}}
\)
up to a constant.
The second integral
\[
	- \int \eta \partial_x (\eta v)
	=
	\frac 12 \int v \partial_x \eta^2
	=
	- \frac 12 \int \eta^2 \partial_x v
	\leqslant
	\frac 12 \lVert \eta^2 \rVert _{L_2} \lVert \partial_x v \rVert _{L_2}
	\leqslant
	C \lVert \eta \rVert _{ H^{1/2} }^2 \lVert v \rVert _{H^1}
\]
where the $L_4$-norm was controlled by $H^{1/2}$-norm
as in Theorem 3.3 of the book by Linares and Ponce \cite{Linares_Ponce}.
Noticing again
$i |D|^2 \tanh D = iD |D| ( |\tanh D| - 1 ) + \partial_x |D|$
one can treat the last part of Integral \eqref{dB_multline} as
\begin{multline*}
	- i \int v | D |^2 \tanh D v^2 / 2
	=
	i \int v D |D| ( 1 - |\tanh D| ) v^2 / 2
	- \int v \partial_x v |D| v
	\leqslant
	\\
	\leqslant
	C \lVert v \rVert _{ L_2 }^2 \lVert v \rVert _{ L_{\infty} }
	+ \lVert \partial_x v \rVert _{ L_2 }^2 \lVert v \rVert _{ L_{\infty} }
	\leqslant
	C \lVert v \rVert _{H^1}^3
	.
\end{multline*}

Thus combining all these inequalities in Identity \eqref{dB_multline}
one arrives to
\[
	\frac {d}{ dt } B(\eta, v)
	\leqslant
	C \left(
		\lVert \eta \rVert _{L_2} \lVert v \rVert _{L_2}
		+
		\lVert \eta \rVert _{L_2}^2 \lVert v \rVert _{H^1}
		+
		\lVert \eta \rVert _{ H^{1/2} }^2 \lVert v \rVert _{H^1}
		+
		\lVert v \rVert _{H^1}^3
	\right)
\]
that is together with \eqref{E_norm} and \eqref{first_dE}
results in
\begin{equation}
\label{dE}
	\frac {d}{ dt } E
	\leqslant
	C(E + E^2)
\end{equation}
where equivalence of $E(\eta, v)$ to
\(
	\lVert \eta \rVert _{H^{1/2}}
	+ \lVert v \rVert _{H^1}
\)
was used.
Integration of \eqref{dE} proves the lemma.

\end{proof}

\section{Uniqueness}
\setcounter{equation}{0}

Suppose on some time interval
we have two solution pairs $\eta_1$, $v_1$ and $\eta_2$, $v_2$
of System \eqref{sys1}-\eqref{sys2} with the same initial data.
Introduce functions $\theta = \eta_1 - \eta_2$, $w = v_1 - v_2$ and
$\zeta = (\eta_1 + \eta_2) / 2$, $u = (v_1 + v_2) / 2$.
Then $\theta$ and $w$ satisfy the following system
\begin{align}
\label{uniqueness_sys1}
	\theta_t &=
	- w_x - i \tanh D (u \theta + \zeta w)
	, \\
\label{uniqueness_sys2}
	w_t &=
	- i \tanh D \theta - i \tanh D (uw)
\end{align}
with zero initial data.
The idea is to obtain an estimate for this system similar
to the priori bound given in the above lemma.
For this purpose one calculates derivative of
the square norm $E^2(\theta, w)$.
Calculations are similar
\begin{equation}
\label{first_uniqueness_dE}
	\frac 12 \frac{d}{dt}
	\left(
		\lVert \theta \rVert _{L_2}^2
		+
		\lVert w \rVert _{L_2}^2
	\right)	
	\leqslant
	\sqrt{2}
	\lVert \theta \rVert _{L_2}
	\lVert w \rVert _{H^1}
	+
	\left(
		\lVert \zeta \rVert _{L_2}
		+
		\lVert u \rVert _{H^1}
	\right)
	\left(
		\lVert \theta \rVert _{L_2}
		+
		\lVert w \rVert _{H^1}
	\right) ^2
\end{equation}
and for the derivative of the rest part of $E^2$ obtain
\begin{multline}
\label{uniqueness_dB_multline}
	\frac {d}{ dt } B(\theta, w)
	=
	\int _{ \mathbb R }
	\left[
		- \theta |D| \partial_x w
	\right.
	-
	\\
	\left.
		- i \theta  |D| \tanh D (u \theta + \zeta w)
		- i w | D |^2 \tanh D \theta
		- i w | D |^2 \tanh D (uw)
	\right]
	dx
	.
\end{multline}
The first and the third integral
in \eqref{uniqueness_dB_multline} together are estimated
exactly as the corresponding part in \eqref{dB_multline}
by
\(
	\lVert \theta \rVert _{L_2} \lVert w \rVert _{L_2}
\)
up to some constant.
Similarly also estimate the fourth integral
in \eqref{uniqueness_dB_multline} by
\(
	\lVert u \rVert _{H^1} \lVert w \rVert _{H^1}^2
\)
up to a constant.
Due to identity
$i |D| \tanh D = \partial_x ( |\tanh D| - 1 ) + \partial_x $,
in stead of regarding the second integral
in \eqref{uniqueness_dB_multline}
it is enough to estimate the following integral
\begin{multline*}
	\int | \theta  \partial_x (\zeta w) |
	\leqslant
	\lVert |\partial_x|^{1/2} \theta \rVert _{L_2}
	\lVert |\partial_x|^{1/2} (\zeta w) \rVert _{L_2}
	\leqslant
	\\
	\leqslant
	C \lVert |\partial_x|^{1/2} \theta \rVert _{L_2}
	\left(	
		\lVert |\partial_x|^{1/2} \zeta \rVert _{L_2}
		\lVert w \rVert _{L_{\infty}}
		+
		\lVert \zeta \rVert _{L_4}
		\lVert |\partial_x|^{1/2} w \rVert _{L_4}
	\right)
	\leqslant
	C \lVert \zeta \rVert _{H^{1/2}}
	\lVert \theta \rVert _{H^{1/2}}
	\lVert w \rVert _{H^1}
\end{multline*}
which finishes the estimation of Derivative
\eqref{uniqueness_dB_multline}.
Firstly, the fractional Leibniz rule was used here,
that was derived by Kenig, Ponce,
and Vega \cite{Kenig_Ponce_Vega}.
For the exact form we apply, one can look
on page 52 of the book by Linares and Ponce \cite{Linares_Ponce}.
Secondly, $L_4$-norms were estimated via $H^{1/2}$-norms.

The resulting inequality has the form
\begin{equation}
\label{uniqueness_dE}
	\frac {d}{ dt } E(\theta, w)
	\leqslant
	C
	\left(
		1 + \lVert \zeta \rVert _{H^{1/2}}
		+ \lVert u \rVert _{H^1}
	\right)
	E(\theta, w)
	.
\end{equation}
Taking into account boundedness
of the norm
\(
	\lVert \zeta \rVert _{H^{1/2}}
	+ \lVert u \rVert _{H^1}
\)
on the regarded time interval
one can deduce uniqueness from
the obtained inequality \eqref{uniqueness_dE}.
%
%
\section{Computation of solitary waves}
\setcounter{equation}{0}
%
%
In this section we calculate numerically solitary waves corresponding
to the Whitham--Bousinesq system
\eqref{sys1}-\eqref{sys2}
and compare them with the Euler solitary waves.
We also regard evolution of Euler solitary waves with respect to
System \eqref{sys1}-\eqref{sys2}.
This comparison supports relevance of
System \eqref{sys1}-\eqref{sys2}
for water waves theory.
It is just an additional justification to what have been done in
\cite{Dinvay_Dutykh_Kalisch}.
%
%
%
%
\begin{figure}[ht!]
	\centering
	\subfigure
	{
		\includegraphics
		[
			width=0.45\textwidth
		]
		{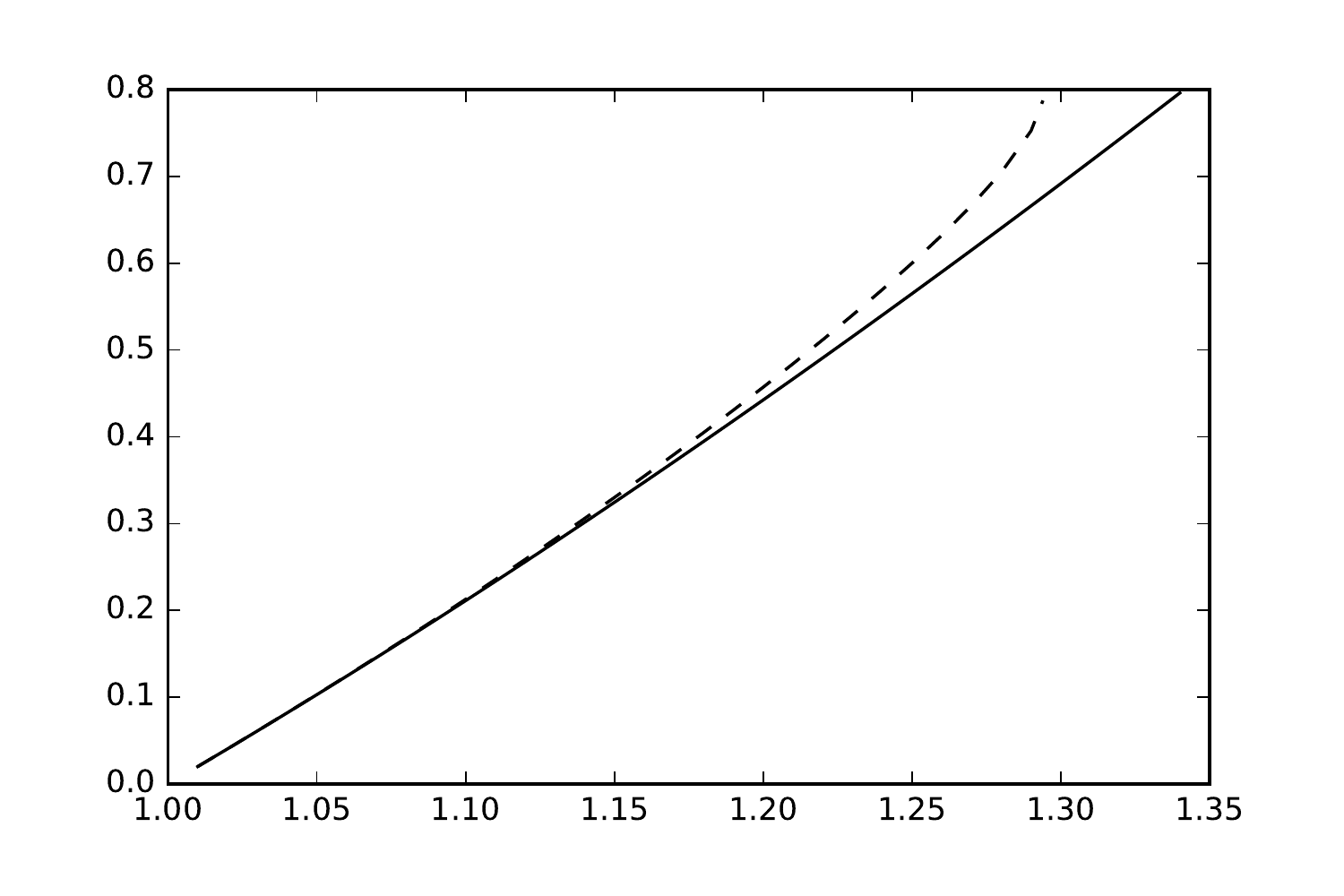}
	}
	~~~~
	\subfigure
	{
		\includegraphics
		[
			width=0.45\textwidth
		]
		{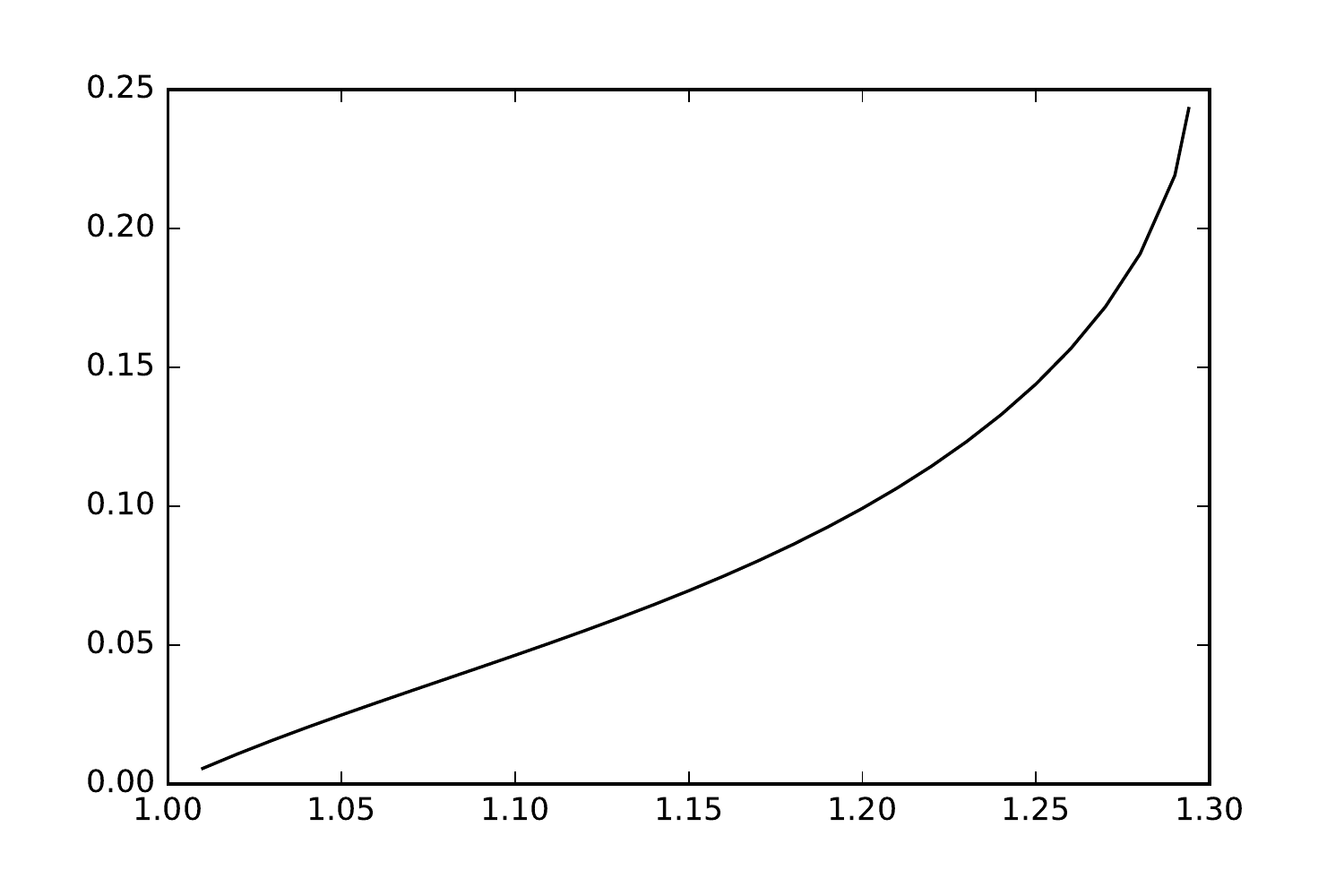}
	}
	\caption
	{
		Amplitude versus speed relation $a(c)$
		for the Whitham system (solid line)
		and the Euler system (dashed line) on the left.
		Relative difference $\epsilon(c)$ on the right.
	}
\label{relation_difference_figure}
\end{figure}
\begin{figure}[ht!]
	\centering
	\subfigure
	{
		\includegraphics
		[
			width=0.45\textwidth
		]
		{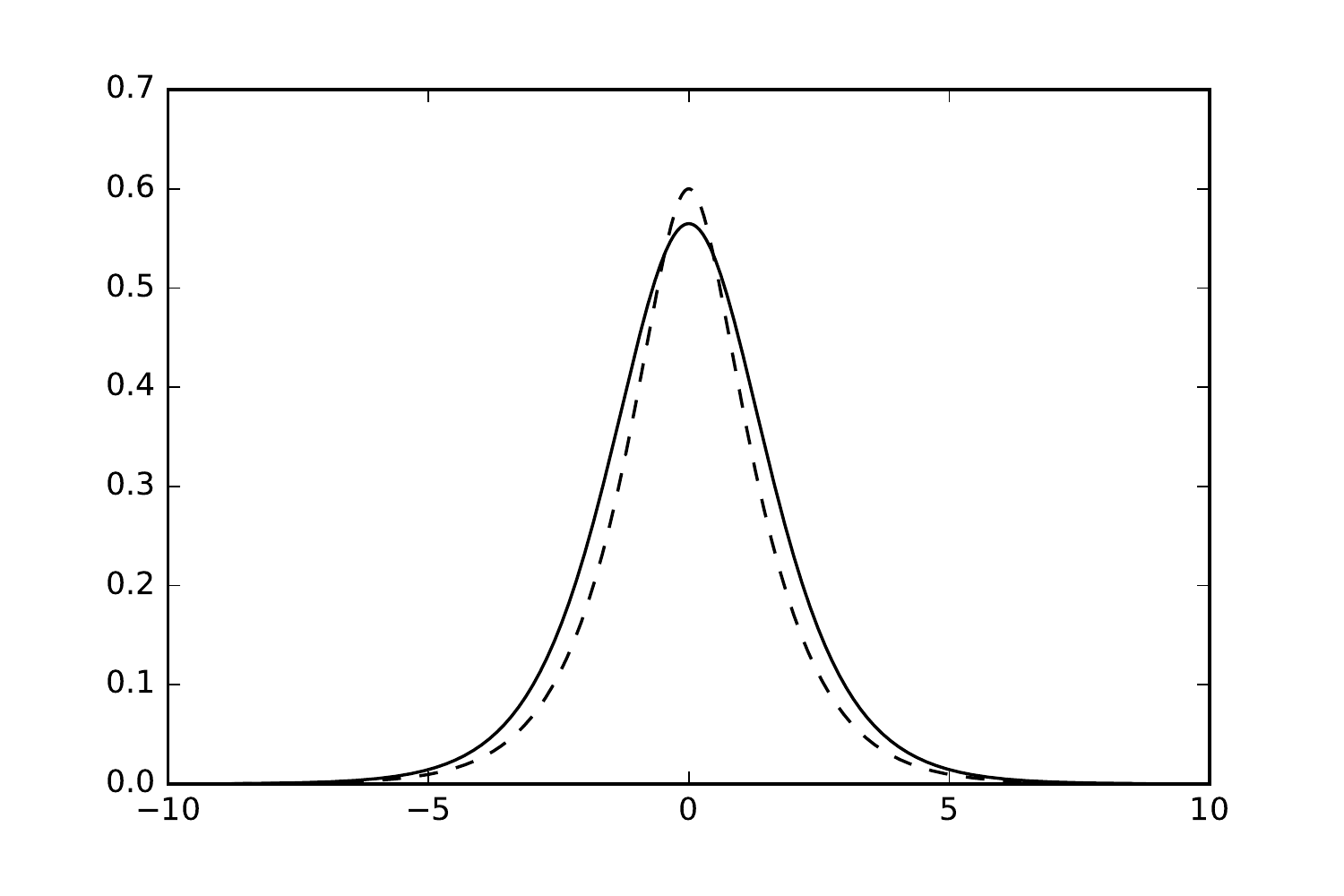}
	}
	~~~~
	\subfigure
	{
		\includegraphics
		[
			width=0.45\textwidth
		]
		{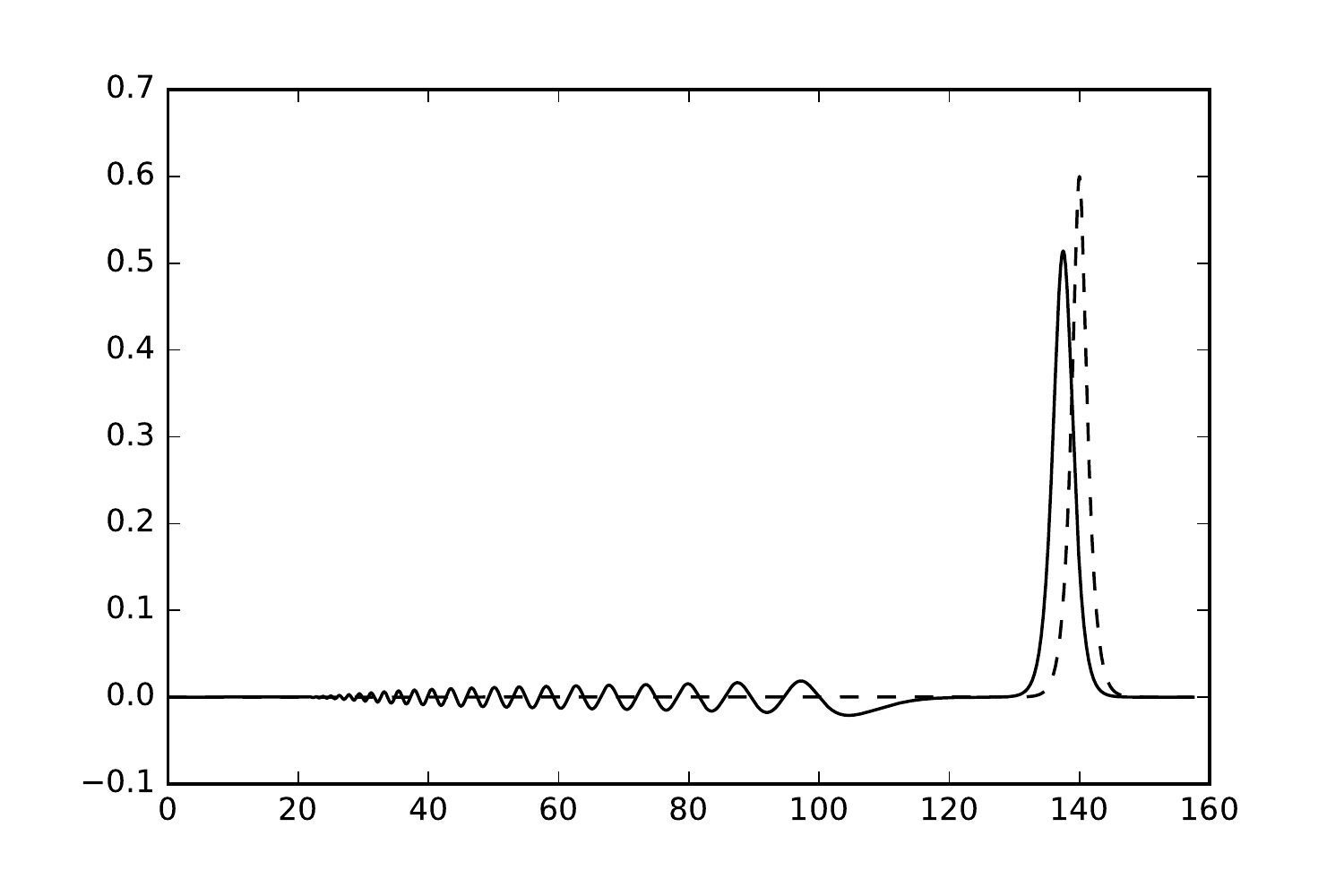}
	}
	\caption
	{
		Solitary waves corresponding the Froude number $c = 1.25$
		for the Whitham system (solid line)
		and the Euler system (dashed line) on the left.
		Evolution of the Euler solitary wave due to the Whitham system
		on the right.
	}
\label{solitons_evolution_figure}
\end{figure}
%
%
%
%

For notational convenience, we use the same notations
$\eta$, $v$ for solitary waves profiles corresponding to
\eqref{sys1}-\eqref{sys2}.
In other words, we write $\eta(x, t) = \eta(x - ct)$ and $v(x, t) = v(x - ct)$.
Here $c$ stands for a Froude number coinciding with the speed of a soliton
in our non-dimensional framework.
The corresponding solitary waves system has the view
\begin{align}
\label{solitary_sys1}
	c\eta &=
	v + \mathcal K (\eta v)
	, \\
\label{solitary_sys2}
	cv &=
	\mathcal K \eta + \mathcal K v^2 / 2
\end{align}
where $\mathcal K = \tanh D / D$ is a bounded self-adjoint operator in
$L_2(\mathbb R)$.
A simple heuristic analysis shows that solutions of
System \eqref{solitary_sys1}-\eqref{solitary_sys2}
are smooth and exist for any $c > 1$.
Indeed, expressing $\eta$ via $v$ by \eqref{solitary_sys2}
and substituting to \eqref{solitary_sys1} one obtains
\[
	v = \frac{1}{c^2} \mathcal K v
	+ \frac{1}{2c} \mathcal K v^2
	+ \frac{1}{c} \mathcal K^2 ( v \mathcal K^{-1} v )
	- \frac{1}{2c^2} \mathcal K^2 v^3
	.
\]
Clearly, operator $\mathcal K^{-1} - |D|$ is bounded and
the operator $\mathcal K$ improves the smoothness of its
operand by one order.
So if one takes $v \in H^1$ and substitute it to the right part
of the last identity then one obviously gets $v \in H^2$,
which results in the fact that both solutions
$v$ and $\eta$ are infinitely smooth and there is no restriction
on their amplitudes.

A use of the Petviashvili iteration method is made to calculate solitary waves \cite{Clamond_Dutykh}.
Applicability of the method is out of scope of this note.
%
%
The essence of the method is to split
the linear $\mathcal L$ and the nonlinear $\mathcal N$ parts
as follows
\[
	\mathcal L(\eta, v)
	=
	\begin{pmatrix}
		c & -1
		\\
		- \mathcal K & c
	\end{pmatrix}
	\begin{pmatrix}
		\eta
		\\
		v
	\end{pmatrix}
	, \quad
	\mathcal N(\eta, v)
	=
	\begin{pmatrix}
		\mathcal K (\eta v)
		\\
		\mathcal K v^2 / 2
	\end{pmatrix}
\]
and so System \eqref{solitary_sys1}-\eqref{solitary_sys2}
can be rewritten as $\mathcal L(\eta, v) = \mathcal N(\eta, v)$.
Clearly, the operator $\mathcal L$ is invertible if and only if
$c^2 > 1$.
The Petviashvili iterative scheme is defined by
\[
	( \eta_{n+1}, v_{n+1} )^T
	=
	S_n^2 \mathcal L^{-1}( \mathcal N( \eta_n, v_n ) )
\]
where $S_n$ is a stabilisation factor computed by
\[
	S_n = \frac
	{ \int ( \eta_n, v_n ) \mathcal L( \eta_n, v_n ) dx }
	{ \int ( \eta_n, v_n ) \mathcal N( \eta_n, v_n ) dx }
	.
\]

An analogous splitting is applied to the Babenko equation
describing Euler gravity solitary surface waves \cite{Clamond_Dutykh}.
This is implemented in the code \cite{Dutykh_code}.
For time evolution performance of
System \eqref{sys1}-\eqref{sys2}, it is treated
by the numerical scheme thoroughly described in \cite{Dinvay_Dutykh_Kalisch}.

For comparison with the fully nonlinear model we introduce
the relative difference between waves $\eta_1$ and $\eta_2$ as
\begin{equation}
\label{wave_difference}
	d( \eta_1, \eta_2) =
	\frac
	{ \lVert \eta_1 - \eta_2 \rVert _{L_2} }
	{ \lVert \eta_1 \rVert _{L_2} }
	.
\end{equation}
As is pointed out above, solutions of
System \eqref{solitary_sys1}-\eqref{solitary_sys2}
are defined for any Froude number $c > 1$,
whereas for the fully nonlinear solitary waves
it does not exceed $c = 1.29421$. 
In Figure \ref{relation_difference_figure} solitons
for different models are compared.
On the left picture
one can see the dependence of amplitude $a = \eta(0)$ on speed $c$.
The black line corresponds to the Whitham--Boussinesq model
and the dashed line to the full Euler model.
On the right picture one can see the dependence on speed of
the relative difference $\epsilon(c) = d( \eta_0, \eta)$,
where Euler $\eta_0$ and Whitham $\eta$ solitons correspond
to the same speed $c$.
It is worth to notice that even for solitary waves with amplitude
of order $a = 0.4$ the error of approximation does not exceed 10\%.
It approaches zero when amplitudes are taken small.

In Figure \ref{solitons_evolution_figure} approximation of
relatively high solitary waves is examined.
On the left picture solitons corresponding to $c = 1.25$
for different models are represented.
The dashed line is for the Euler solitary wave $\eta_0(x)$.
The latter is taken as an initial condition for numerical integration
of System \eqref{sys1}-\eqref{sys2}.
Thus one can look at the time evolution of the fully nonlinear
solitary wave with respect to the approximate model
System \eqref{sys1}-\eqref{sys2}
on the right picture in Figure \ref{solitons_evolution_figure}.
The shot is taken at the moment $t = 112$.
The corresponding initial data has the form
\[
	\eta(x, 0) = \eta_0(x)
	, \quad
	v(x, 0) = \mathcal K( u_1 + u_2 \partial_x \eta_0 )
\]
where elevation $\eta_0$, horizontal $u_1$ and vertical $u_2$ velocities
are associated the Euler solitary wave moving with the speed $c = 1.25$
(the dashed line on the picture).
One can see that the initial wave is diminishing leaving
a dispersive tail behind.
It is worth to notice that after some time this leading wave
turns out to be a solitary solution
of \eqref{solitary_sys1}-\eqref{solitary_sys2}.
More precisely, if one excludes the tail from the solution $\eta(x, t)$
then at the moment $t = 10$ minutes
(according to our nondimensional settings)
we have the difference $d(\eta_s, \eta) = 2.2 \cdot 10^{-5}$.
Here $\eta_s$ is the solution of
\eqref{solitary_sys1}-\eqref{solitary_sys2}
corresponding to the Froude number $c = 1.22957$.
This allows us to make a conjecture about asymptotic stability
of solitary waves for the regarded model
\eqref{sys1}-\eqref{sys2}.
%
%
\section{Conclusions}
\setcounter{equation}{0}
%
%
The dispersive Boussinesq system \eqref{sys1}-\eqref{sys2}
was derived using Hamiltonian perturbation theory
by Dinvay, Dutykh and Kalisch \cite{Dinvay_Dutykh_Kalisch}.
In the current paper this system has been proved to be locally
well-posed.
Its accuracy as of an asymptotic model was tested with
solitary waves, the latter admit a complete characterization
via speed-amplitude relation.

There are many possibilities for further study of
System \eqref{sys1}-\eqref{sys2}.
First, it is desirable to prove rigorously consistency.
Second, it is of interest to check if the model
features modulational instability and wave breaking.
Third, it would be interesting to try to extend the local result
of the paper to a global well-posednes and
possibly to prove asymptotic stability of solitary waves.


\vskip 0.05in
\noindent
{\bf Acknowledgments.}
{
The author is grateful to Didier Pilod and Henrik Kalisch
who read the manuscript and made some comments.
}

\vskip -0.1in

\end{document}